\documentclass[12pt]{amsart}

\usepackage{amsmath,amssymb,amscd,amsthm,graphicx,enumerate}

\usepackage{hyperref}
\hypersetup{breaklinks=true}

\numberwithin{equation}{section}
\theoremstyle{plain}

\makeatletter
\newtheorem*{rep@theorem}{\rep@title}
\newcommand{\newreptheorem}[2]{%
\newenvironment{rep#1}[1]{%
 \def\rep@title{#2 \ref{##1}}%
 \begin{rep@theorem}}%
 {\end{rep@theorem}}}
\makeatother

\newtheorem{theorem}[equation]{Theorem}
\newreptheorem{theorem}{Theorem}

\newtheorem{proposition}[equation]{Proposition}
\newtheorem{lemma}[equation]{Lemma}
\newtheorem{corollary}[equation]{Corollary}

\newtheorem{claim}[equation]{Claim}

\theoremstyle{remark}
\newtheorem{remark}[equation]{Remark}

\theoremstyle{definition}
\newtheorem{definition}[equation]{Definition}

\newtheorem*{question*}{Question}

\newcommand{\K}{{\mathcal K}}

\newcommand{\R}{\mathbf R}

\newcommand{\diam}{\operatorname{Diam}}

\newcommand{\al}{\alpha}

\newcommand{\D}{\partial}

\newcommand{\Lap}{\Delta}

\newcommand{\eps}{\varepsilon}

\newcommand{\ra}{\rightarrow}

\providecommand{\abs}[1]{\lvert #1\rvert}

\def\XXint#1#2#3{{\setbox0=\hbox{$#1{#2#3}{\int}$}
     \vcenter{\hbox{$#2#3$}}\kern-.5\wd0}}

\begin{document}

\title[Star-shaped mean curvature flow]{Mean curvature flow of star-shaped hypersurfaces}

\author{Longzhi Lin}
\address[L. ~Lin]{Mathematics Department\\University of California - Santa Cruz\\1156 High Street, Santa Cruz, CA 95064\\USA}
\email{lzlin@ucsc.edu}

\thanks{The author was partially supported by a Faculty Research Grant awarded by the Committee on Research from UC, Santa Cruz.}

%    \subjclass is required.
%    The 2010 edition of the Mathematics Subject Classification is
%    now available.  If you are citing a classification from the
%    new scheme, use the following input coding instead.

\date{}
\maketitle

%\tableofcontents

\begin{abstract}
In 1998  Smoczyk \cite{Smoczyk_star_shaped} showed that, among others, the blowup limits at singularities are convex for the mean curvature flow starting from a closed star-shaped surface in $\mathbf{R}^3$. We prove in this paper that this is true for the mean curvature flow of star-shaped hypersurfaces in $\mathbf{R}^{n+1}$ in arbitrary dimension $n\geq 2$. In fact, this holds for a much more general class of initial hypersurfaces. In particular, this implies that the mean curvature flow of star-shaped hypersurfaces is generic in the sense of Colding-Minicozzi \cite{CM_generic}.
\end{abstract}

\section{Introduction}
A family of hypersurfaces evolves by mean curvature flow if the velocity at each point is given by the mean curvature vector. Mean curvature flow has been extensively studied ever since the pioneering work of Brakke \cite{brakke} and Huisken \cite{Huisken_convex}.
While the theory was progressing in many fruitful directions, 
there was one persistent central theme:  
the investigation of singularities, and the development of
 related techniques.  In the last 15 years, this
culminated in the spectacular
work of White \cite{white_size,white_nature,white_subsequent} and Huisken-Sinestrari 
\cite{huisken-sinestrari1,huisken-sinestrari2,huisken-sinestrari3}
on mean curvature flow in the case of mean convex 
hypersurfaces, i.e. hypersurfaces with positive mean curvature.
Their papers give a far-reaching structure theory, providing a package of estimates
that yield a qualitative picture of singularities and a global description of the
large curvature part in a mean convex flow.

In a recent paper \cite{haslhofer-kleiner_mean_convex} (see also \cite{haslhofer-kleiner_surgery}), Haslhofer-Kleiner gave a new treatment of the theory of White and Huisken-Sinestrari. A key ingredient in this new approach is a new preserved quantity under mean convex mean curvature flow discovered by Andrews \cite{andrews1} (see also \cite{white_size,sheng_wang}),
called $\alpha$-noncollapsing.
A mean convex hypersurface $M\subset\R^{n+1}$ is \emph{$\alpha$-noncollapsed},
 if each point $p\in M$ admits interior and exterior ball tangent at $p$ of radius at least $\alpha/H(p)$.
The definition and preservation of $\alpha$-noncollapsing crucially depends on the fact that the the mean curvature is positive ($H>0$).

If the initial hypersurface $M_0$ is not mean convex, then the theory of mean convex mean curvature flow is not applicable. It is thus a very interesting question, whether the results can nevertheless be extended to some situations where the mean curvature changes sign.
As observed by Smoczyk \cite{Smoczyk_star_shaped}, a good situation to look for such extensions is the setting where the initial hypersurface is \emph{star-shaped}, i.e. where $M_0$ satisfies $\langle X,\boldsymbol{\nu}\rangle>0$, with $\boldsymbol{\nu}$ denoting the outward unit normal.
In this setting, the relevant quantity to consider is $F=\langle X,\boldsymbol{\nu}\rangle +2tH$, which is nondecreasing and positive along the flow.
Smoczyk proved that the Huisken-Sinestrari convexity estimate holds for the flow of star-shaped surfaces in $\mathbb{R}^3$ \cite[Thm. 1.1]{Smoczyk_star_shaped},
and it was pointed out by Huisken-Sinestrari \cite[Rem. 3.8]{huisken-sinestrari2} that (by taking care of some lower order terms) their proof of the convexity estimate in fact goes through for star-shaped hypersurfaces in arbitrary dimension.
Moreover, it has been observed by Andrews \cite{andrews1} that a variant of his $\alpha$-noncollapsing condition, where $H$ is replaced by $F$, is preserved for mean curvature flow with star-shaped initial condition.

In this paper, we will use the framework of Haslhofer-Kleiner \cite{haslhofer-kleiner_mean_convex} for mean convex mean curvature flow to prove estimates and structural results for the star-shaped case. In fact, our results are true for a much more general class of initial hypersurfaces, see Remark \ref{general_class}.  In Section \ref{sec_prelim}, we collect some preliminaries on (star-shaped) mean curvature flow and recall the variant of Andrews' noncollapsing result for the star-shaped case (Theorem \ref{thm_noncollapsing}).
In Section \ref{sec_mainest}, we prove three main estimates for the mean curvature flow with star-shaped initial condition.
The local curvature estimate (Theorem \ref{thm-local_curvature_est}) gives curvature control in a parabolic neighborhood of definite size assuming only curvature control at a single point.
The convexity estimate (Theorem \ref{thm_convexity_est}) gives pinching of the principal curvatures towards positive.
The blowup theorem (Theorem \ref{thm_blow_up_I}) allows us to pass to blowup limits smoothly and globally.
In Section \ref{sec_weaksol}, we explain that our three main estimates still hold beyond the first singular time
if the mean curvature is interpreted in the viscosity sense (Definition \ref{def_viscosity}).
As a consequence, we obtain a structure theorem (Theorem \ref{thm_tangent_flow}),
which says that all tangent flows in the star-shaped case are either planes or shrinking round spheres or cylinders,
and a partial regularity theorem (Theorem \ref{thm_partial_regularity}),
which says that the parabolic Hausdorff (and Minkowski) dimension of the singular set $\mathcal{S}\subset\R^{n+1,1}$ in the star-shaped case is at most $n-1$.
In particular, we see that the star-shaped case provides a setting where all mean curvature flow singularities are generic in the sense of \cite{CM_generic}.
Thus, applying recent results of Colding-Minicozzi \cite{CM_unique,CM_singular_set} we can also conclude that all tangent flows are unique and that the $(n-1)$-dimensional parabolic Hausdorff measure of $\mathcal{S}$ is in fact finite.

The proofs of the local curvature estimate, the convexity estimate and elliptic regularization for star-shaped flows
are quite different from \cite{haslhofer-kleiner_mean_convex} and require a number of new ideas.
E.g. we have to relate bounds for $F$ (which appears in the definition of noncollapsing)
and bounds for $H$ (which we get by comparison with spheres),
and we have to overcome the difficulty that for star-shaped flow the speed $H$ doesn't have a sign, i.e. that the motion in general doesn't produce a foliation.

\noindent{\textit{Acknowledgement.}} The author would like to thank Robert Haslhofer for the continued stimulating and useful discussions.

\section{Preliminaries}\label{sec_prelim}
\subsection{Notation and terminology}
A smooth family $\{M_t \subset \R^{n+1}\}_{t\in I}$ of closed embedded hypersurfaces, where $I\subset \R$ is an interval, \textit{moves by mean curvature flow} if $M_t = X_t(M) = X(M, t)$ for some smooth family of embeddings $\{X_t: M \to \R^{n+1}\}_{t\in I}$ satisfying the mean curvature flow equation
$$
\frac{\partial X_t}{\partial t} = - H \boldsymbol{\nu},
$$ 
where $H$ denotes the mean curvature and $\boldsymbol{\nu}$ is the outward unit normal at $X_t$. Instead of the family $\{M_t\}$ itself, we will think in terms of the evolving family $\{K_t\}$ of the compact domains bounded by the $M_t$'s. 

\textit{Space-time} $\R^{n+1, 1}$ is defined to be $\R^{n+1} \times \R$ equipped with the parabolic metric $d((x_1,t_1),(x_2,t_2)) = \max(|x_1 - x_2|, |t_1 -t_2|^{\frac{1}{2}})$. \textit{Parabolic rescaling by $\lambda \in (0,\infty)$ at $(x_0,t_0) \in \R^{n+1, 1}$} is described by the mapping
$$
(x,t) \mapsto (\lambda (x-x_0), \lambda^2(t-t_0)).
$$
The \textit{parabolic ball} with radius $r>0$ and center $X=(x,t) \in \R^{n+1,1}$ is the product
$$
P(x,t,r) = B(x,r) \times (t-r^2, t] \subset \R^{n+1,1}\,.
$$
When we talk about a flow in a parabolic ball $P(x,t,r)$ we in particular include the assumption that the flow existed at least since $t-r^2$.

Given a family of subsets $\{K_t \subseteq \R^{n+1}\}_{t\in I}$ its \textit{space-time track} is the set
$$
\mathcal{K} = \cup_{t\in I} K_t \times \{t\} \subseteq \R^{n+1, 1}\,.
$$
Given a subset $\mathcal{K} \subseteq \R^{n+1,1}$, the \textit{time $t$ slice} of $\mathcal{K}$ is
$$
K_t = \{x\in \R^{n+1} \,|\, (x,t) \in \mathcal{K}\}\,.
$$ 

Given a smooth compact domain $K_0\subset \R^{n+1}$ we write $K_t$ for the evolution of $K_0$ by mean curvature flow. In technical terms, this is the \textit{level set flow} $\{K_t \subset \R^{n+1}\}$ starting at $K_0$, see \cite{evans-spruck}, \cite{CGG} and \cite{Ilmanen}. The level set flow can be defined as the maximal family of closed sets $\{K_t\}_{t\geq 0}$ starting at $K_0$ that satisfies the the avoidance principle
$$
K_{t_0} \cap L_{t_0} = \emptyset \Rightarrow K_{t} \cap L_{t} = \emptyset \quad \text{for all }\, t\in [t_0,t_1], 
$$
whenever $\{L_t\}_{t\in [t_0,t_1]}$ is a smooth compact mean curvature flow.
The definition is phrased in such a way, that existence and uniqueness are immediate. Moreover, the level set flow of $K_0$ coincides with smooth
mean curvature flow of $K_0$ for as long as the latter is defined.

We suppress the dependence on $n$ in the notation, and we always assume that the initial domain $K_0\subset \R^{n+1}$ is smooth and compact.

%%%%%%%%%%%%%%%%%%%%%%%%%%%%%%%%%%%%

\subsection{Star-shapedness and $\alpha$-noncollapsing}
A smooth compact domain $K_0\subset \R^{n+1}$ is called \emph{star-shaped} (around the origin) if $\langle X,\boldsymbol{\nu} \rangle >0$ for all $X\in\partial K_0$.

\begin{proposition}[{\cite[Prop. 4]{Smoczyk_star_shaped}}]
The quantity $F=\langle X,\boldsymbol{\nu}\rangle +2tH$ satisfies the evolution equation
\begin{equation}\label{eqn_evo_F}
\partial_t F=\Lap F+\abs{A}^2F,
\end{equation}
In particular, if $K_0$ is star-shaped, then $F$ is positive for all $t\geq 0$ as long as the flow exists.
\end{proposition}

\begin{proposition}[{c.f. \cite[Lem. 1.1]{Smoczyk_star_shaped}}]\label{cor_H_lower_bound}
If $K_t$ is a mean curvature flow starting at a star-shaped domain $K_0\subset\R^{n+1}$,
then $H\geq -C$, where $C$ only depends on $\beta = \max\{\max_{\partial K_0} |A|, \diam(\partial K_0)\}$.
\end{proposition}

\begin{proof} Using the evolution equation for $|A|^2$,
$$
\partial_t |A|^2 = \Delta |A|^2 - 2|\nabla A|^2 + 2 |A|^4,
$$
and the maximum principle, there exists some small $\sigma>0$ depending only on $\max_{\partial K_0} |A|$,
such that for all $t \in [0,\sigma]$ we have $\max_{\partial K_t} |A| \leq 2\max_{\partial K_0} |A|$\,.
This gives $\tilde{C}$ such that $H \geq -\tilde{C}$ for all $t \in [0,\sigma]$. For $t> \sigma$, since $F=\langle X,\boldsymbol{\nu}\rangle +2tH$ is positive we have $H\geq -\diam(\partial K_0)/\sigma$.
\end{proof}

\begin{remark}\label{general_class}
In fact, the quantity $F = a_1\langle X,\boldsymbol{\nu}\rangle +(a_2 + 2a_1 t)H$, where $a_1+a_2>0$, also satisfies equation \eqref{eqn_evo_F}. Therefore, with minor modifications (cf. \eqref{eqn_rescaling} - \eqref{eqn_rescaled_flow_mono}), our proofs generalize to the class of initial hypersurfaces that satisfy the condition $a_1\langle X,\boldsymbol{\nu}\rangle +a_2H >0$. Such class of initial hypersurfaces include mean convex hypersurfaces ($a_1 = 0, a_2 =1$) and star-shaped hypersurfaces ($a_1 = 1, a_2 = 0$) and more, cf. \cite{Smoczyk_star_shaped}.
\end{remark}

\begin{definition}[{c.f. \cite[Def. 1]{andrews1}}]\label{def_noncoll}
Let $\alpha>0$. A smooth compact domain $K\subset \R^{n+1}$ with $F>0$ is \emph{$\alpha$-noncollapsed} if each point $p\in \partial K$ admits interior and exterior balls tangent at $p$ of radius at least $\alpha/{F(p)}$.
\end{definition}

By compactness, each star-shaped domain $K_0\subset \R^{n+1}$ satisfies the $\alpha$-noncollapsing condition for some $\alpha=\alpha(K_0)>0$.
The following theorem shows that $\alpha$-noncollapsing is preserved along the mean curvature flow.

\begin{theorem}[c.f. {\cite[Rem. 7]{andrews1}, \cite[Rem. 3]{Andrews_Langford_McCoy}}]\label{thm_noncollapsing}
If $K_0$ is $\alpha$-noncollapsed, then $K_t$ is $\alpha$-noncollapsed for the same constant $\alpha$.
\end{theorem}

\begin{proof}
The proof follows from a similar computation as in \cite{andrews1} and \cite{Andrews_Langford_McCoy}. For the convenience of reader, we include it here. Consider
$$Z(x,y,t)= \frac{2\langle X(y,t) - X(x,t), \boldsymbol{\nu}(x,t) \rangle}{\| X(y,t) - X(x,t)\|^2}$$
and
$$Z_\ast(x,t)= \inf_{y\neq x} Z(x,y,t)\,,\quad Z^\ast (x,t)= \sup_{y\neq x} Z(x,y,t)\,.$$
By a simple geometric argument, interior and exterior $\alpha$-noncollapsing is equivalent to the inequalities $\frac{Z_\ast}{F} \geq -\frac{1}{\alpha}$ and $\frac{Z^\ast}{F} \leq \frac{1}{\alpha}$, respectively.

Computing various derivatives of $Z$, Andrews-Langford-McCoy derived the evolution inequalities (in the viscosity sense),
\begin{equation}\label{eqn_Z_ast}
\partial_t Z_\ast \geq \Delta Z_\ast + |A|^2 Z_\ast\,,\quad \partial_t Z^\ast \leq \Delta Z^\ast + |A|^2 Z^\ast\,,
\end{equation}
see \cite[Thm. 2]{Andrews_Langford_McCoy}.
 Combining this with \eqref{eqn_evo_F} we obtain
\begin{align}\label{eqn_ALM}
 (\partial_t - \Delta) \frac{Z_\ast}{F} &= \frac{(\partial_t - \Delta) Z_\ast}{F} - \frac{Z_\ast (\partial_t - \Delta) F}{F^2} + 2 \left\langle \nabla \log F, \nabla \frac{Z_\ast}{F}\right\rangle\nonumber\\
 &\geq 2 \left\langle \nabla \log F, \nabla \frac{Z_\ast}{F}\right\rangle\,.
 \end{align}
By the maximum principle, the minimum of $\frac{Z_\ast}{F}$ is nondecreasing in time. In particular, if the inequality $\frac{Z_\ast}{F} \geq -\frac{1}{\alpha}$ holds at $t = 0$, then this inequality holds for all $t$. Arguing similarly we obtain that
 \begin{equation}
  \partial_t \frac{Z^\ast}{F} \leq \Delta \frac{Z^\ast}{F} + 2 \left\langle \nabla \log F, \nabla \frac{Z^\ast}{F}\right\rangle\,,
 \end{equation}
and thus that the inequality $\frac{Z^\ast}{F} \leq \frac{1}{\alpha}$ is also preserved along the flow.
\end{proof}

\begin{remark}[parabolic rescaling]\label{remark_rescaling}
If $\{K_t\}_{t\in I}$ is an $\alpha$-noncollapsed flow and if $\{\hat{K}_t\}_{t\in \hat{I}}$ denotes the flow obtained by the parabolic rescaling $(x,t) \to (\lambda x, \lambda^2 t)$, $\lambda \in (0,\infty)$, then $\{\hat{K}_t\}_{t\in \hat{I}}$ is $(\lambda^2\alpha)$-noncollapsed.
\end{remark}

%%%%%%%%%%%%%%%%%%%%%%%%%%%%%%%%%%%%%%%%%%%%%%%
\section{Main estimates and consequences}\label{sec_mainest}

Throughout this section, we consider mean curvature flows $\{K_t\}$ starting at a smooth compact star-shaped initial domain $K_0\subset\R^{n+1}$.
We denote by $\alpha=\alpha(K_0)>0$ and $\beta=\beta(K_0)>0$ the constants from Definition \ref{def_noncoll} and Proposition \ref{cor_H_lower_bound}, respectively.
In this section, we give the proofs in the smooth setting; we refer to Section \ref{sec_weaksol} for the extension of the results to the setting of weak solutions (level set flow).

\subsection{Local curvature estimate}
Our first main estimate gives curvature control on a parabolic ball of definite size, from a bound on the mean curvature $H$ at a single point.

\begin{theorem}[Local curvature estimate]\label{thm-local_curvature_est}
There exist $\rho=\rho(\al, \beta)>0$ and $C_l=C_l(\al,\beta)<\infty$ with the following property.
If $\mathcal{K}$ is a mean curvature flow
with star-shaped initial condition, defined
 in a parabolic ball $P(p,t,r)$ centered at a boundary point $p\in \D K_t$ with $H(p,t)\leq r^{-1}$, then $\mathcal{K}$ is smooth in the parabolic ball $P(p,t,\rho r)$, and
\begin{equation}\label{eqn-intro_curvature_estimate}
 \sup_{P(p,t,\rho r)}\abs{ \nabla^l A}\leq C_l r^{-(l+1)}\,.
\end{equation}
\end{theorem}

As an immediate consequence of Theorem \ref{thm-local_curvature_est}, we obtain:

\begin{corollary}[Gradient estimate]
Suppose $\mathcal{K}$ is a mean curvature flow with star-shaped initial condition. Then we have the gradient estimate
$$
|\nabla A| \leq C H^2\,,
$$
where $C=C(K_0)<\infty$.
\end{corollary}

\begin{proof}[{Proof of Theorem \ref{thm-local_curvature_est}}] Fix $\alpha$ and $\beta$. We will show that there exists a $\rho'>0$ such that the estimate \eqref{eqn-intro_curvature_estimate} holds for $l=0$ with $C_0 = \frac{1}{\rho'}$; the higher order derivative estimates then follow immediately from standard interior estimates (see e.g. \cite[Prop. 3.22]{Ecker_book}).

Suppose this doesn't hold. Then there are sequences of $\alpha$-noncollapsed flows $\left\{\mathcal{K}^j\right\}$,
boundary points $\{p_j \in \partial K_{t_j}\}$ and scales $\{r_j\}$, such that $\mathcal{K}^j$ is defined in $P(p_j,t_j,r_j)$ and $H(p_j, t_j)\leq r_j^{-1}$, but $\sup_{P(p_j,t_j,j^{-1}r_j)} |A|\geq j r_{j}^{-1}$.
After parabolically rescaling by $jr_j^{-1}$ and applying an isometry, we obtain a sequence $\{\hat{\mathcal{K}}^j\}$ of mean curvature flows defined in $P(0,0, j)$ with $H(0,0)\leq j^{-1}$, but
\begin{equation}\label{eqn_A_lower_bound}
\sup_{P(0,0,1)} |A| \geq 1\,.
\end{equation}
Moreover, we can choose coordinates such that the outward normal of $\hat{K}_0^j$ at $(0,0)$ is $e_{n+1}$. 
\begin{claim}[Halfspace convergence]\label{cl_halfspace_convergence}
The sequence of mean curvature flows $\{\hat{\mathcal{K}}^j\}$ converges in the pointed Hausdorff topology to a static halfspace in $\R^{n+1} \times (-\infty,0]$, and similarly for their complements.
\end{claim}
\begin{proof}[Proof of Claim \ref{cl_halfspace_convergence}] For $R < \infty, d>0$ let $\bar{B}_{R,d} = \overline{B((-R+d)e_{n+1}, R)}$ be the closed $R$-ball tangent to the horizontal hyperplane $\{x_{n+1} = d\}$ at the point $de_{n+1}$.  When $R$ is large, it will take time approximately $dR$ for $\bar{B}_{R,d}$  to leave the upper half space $\{x_{n+1} > 0\}$.

Since $0\in \partial \hat{K}_0^j$ for all $j$,  it follows that $\bar{B}_{R,d}$ cannot be contained in the interior of $\hat{K}_t^j$ for any $t\in [-T,0]$, where $T \simeq dR$.
Thus, for large $j$ we can find $d_j \leq d$ such that  $\bar{B}_{R,d_j}$ has interior contact with $\hat{K}_t^j$ at some point $\hat{q}_j$, where $\langle \hat{q}_j, e_{n+1}\rangle < d$ and $\|\hat{q}_j\|\lesssim \sqrt{dR}$.

The mean curvature of $\partial \hat{K}_t$ satisfies $\hat{H}(\hat{q}_j,t) \leq \frac{n}{R}$,
and therefore for $\partial K_t$ we have $H(q_j,s) \leq \frac{n j}{R r_j}$ where $s = (j^{-1}r_j)^2 t + t_j$.
Moreover, by avoidance principle for the mean curvature flow it is clear that
\begin{equation}\label{eqn_extinction_time}
s\leq \frac{D^2}{2n}\,,
\end{equation}
where $D$ is the diameter of $\partial K_0$.
Thus
\begin{equation}
F(q_j,s) \leq D +  \frac{ jD^2}{R r_j} \leq \frac{ 2jD^2}{R r_j}\,,
\end{equation}
provided $R\leq jr_j^{-1}D$. Since $K^j_t$ satisfies the $\al$-noncollapsing condition, there is a closed ball $\bar B_{j,o}$ with radius at least $\frac{\al R r_j}{2jD^2}$ making exterior contact with $K^j_s$ at $q_j$. Therefore, after rescaling, there is a closed ball $\bar B_j$ with radius at least $\frac{\al R}{2D^2}$ making exterior contact with $\hat{K}^j_0$ at $\hat{q}_j$.  By a simple geometric calculation, this implies that $\hat{K}^j_t$ has  height  $\lesssim \frac{D^2d}{\al} $ in the ball $B(0,R')$ where $R'$ is comparable  to $\sqrt{dR}$.  As $d$ and $R$ are arbitrary (in fact, $R$ is allowed to be larger and larger as $j$ increases provided $R\leq jr_j^{-1}D$; also note that $r_j^2\leq D^2/2n$), this implies that for any $T>0$,  and any compact subset $Y\subset\{x_N>0\}$, for large $j$ the time slice $\hat{K}^j_t$ is disjoint from $Y$, for all $t \geq -T$.   Likewise, for any $T>0$ and any compact subset $Y\subset \{x_N<0\}$, the time slice $\hat{K}^j_t$ contains $Y$ for all $t \in[-T, 0]$, and large $j$, because $\hat{K}^j_{-T}$ will contain a ball whose forward evolution under mean curvature flow contains $Y$ at any time $t\in [-T, 0]$. This proves the claim. \end{proof}

To finish the proof of the theorem, we need a variant of the one-sided minimization theorem, cf. \cite[Thm. 3.5]{white_size}, \cite[Rem. 2.6]{haslhofer-kleiner_mean_convex}. 

\begin{claim}[One-sided minimization for $\hat{K}_t^j$]\label{thm_one_sided}
For every $\varepsilon>0$, every $t\in [-T, 0]$ and every ball $B(x,r)$ centered on the hyperplane $\{x_{n+1} = 0\}$, we have
\begin{equation}
|\partial \hat{K}_{t}^j \cap B(x,r)| \leq (1+ \varepsilon)\omega_n r^n\,,
\end{equation}
for $j$ large enough.
\end{claim}

Combining Claim \ref{cl_halfspace_convergence}, Claim \ref{thm_one_sided} and the local regularity theorem for the mean curvature flow (see e.g. \cite{white_regularity}, \cite{Wang_regularity}), we see that $\{\hat{\mathcal{K}}^j\}$ converges
smoothly on compact subsets of spacetime to a static halfspace. In particular,
$$\limsup_{j\to \infty}\sup_{P(0,0,1)} |A| = 0 ;$$ this contradicts \eqref{eqn_A_lower_bound}. Modulo the proof of Claim \ref{thm_one_sided}, which we will give below, this concludes the proof of Theorem \ref{thm-local_curvature_est}.\end{proof}

To prove Claim \ref{thm_one_sided}, we will rescale the flow and prove a weighted version of the one-sided minimization result for the rescaled flow, and then convert it back to the original flow. The key is to make use of the fact that $F= \langle X, \boldsymbol{\nu} \rangle + 2tH>0$ along the flow. We first perform the continuous rescaling:
\begin{equation}\label{eqn_rescaling}
\tilde{X}(\cdot, \tau) = \frac{1}{\sqrt{t}} X(\cdot, t)\,, \quad \tau = \log t\,.
\end{equation}
Then $\tilde{X}(\cdot, \tau)$ satisfies the \textit{rescaled mean curvature flow equation}
\begin{equation}\label{eqn_rescaled_flow}
\left( \frac{\partial }{\partial \tau} \tilde{X} \right)^{\perp} = - \left( \tilde{H} + \frac{\langle \tilde{X}, \tilde{\boldsymbol{\nu}}\rangle}{2}\right) \tilde{\boldsymbol{\nu}}\,, \quad \tau \in (-\infty, \log T)\,.
\end{equation}
Note that the speed function on the right-hand side of \eqref{eqn_rescaled_flow} is negative:
\begin{equation}\label{eqn_rescaled_flow_mono}
- \left( \tilde{H} + \frac{\langle \tilde{X}, \tilde{\boldsymbol{\nu}}\rangle}{2}\right) = - \sqrt{t}\left(H + \frac{\langle X, \boldsymbol{\nu}\rangle}{2t}\right) < 0\,.
\end{equation}

Let $\{K_t=K_t^j\}$ be the sequence of flows from the proof of Theorem \ref{thm-local_curvature_est},
and denote by $\{\tilde{K}_\tau\}$ the associated rescaled mean curvature flows (we suppress the index $j$ in the notation).
Using \eqref{eqn_rescaled_flow} and \eqref{eqn_rescaled_flow_mono}, we see that the boundaries of the $\{\tilde{K}_{\tau}\}_{\tau\in (-\infty, \log T)}$ form a foliation of $\R^{n+1} \backslash \tilde{K}_{\log T}$ for any $T>0$ as long as the flow exits.

Now we define the \textit{weighted boundary area} of a compact set $S$ with sufficiently regular boundary to be
$$
\text{Area}_w(\partial S) = \int_{\partial S} e^{\frac{|x|^2}{4}} d\mu\,.
$$
Note that if $\partial S$ minimizes the weighted boundary area, then on $\partial S$ we have
$$
H + \frac{\langle X,\boldsymbol{\nu}\rangle}{2} = 0\,.
$$
\begin{claim}[Weighted one-sided minimization for rescaled flow]\label{thm_one_sided_rescaled}
The weighted boundary area of $\tilde{K}_{\tau}$ is less than or equal to the weighted boundary area of any smooth compact domain $S\supseteq\tilde{K}_{\tau}$.
\end{claim}
\begin{proof}[{Proof of Claim \ref{thm_one_sided_rescaled}}]
Recall that $\{\partial\tilde{K}_{\tau'}\}_{\tau'\leq \tau}$ foliates $\R^{n+1} \backslash \text{Int}(\tilde{K}_{\tau})$. Let $\tilde{\boldsymbol{\nu}}$ be the vector field in $\R^{n+1} \backslash \text{Int}(\tilde{K}_{\tau})$ defined by the outward unit normals of the foliation.
If $S\supseteq\tilde{K}_{\tau}$ is any smooth compact domain, then using the divergence theorem we can compute
\begin{align*}
&\text{Area}_w(\partial S) -  \text{Area}_w(\partial  \tilde{K}_{\tau}) \\
  &\quad\qquad \geq \int_{\partial S } \langle \tilde{\boldsymbol{\nu}}, \tilde{\boldsymbol{\nu}}_{\partial S}\rangle e^{\frac{|x|^2}{4}} d\mu -   \int_{\partial \tilde{K}_{\tau} } \langle \tilde{\boldsymbol{\nu}}, \tilde{\boldsymbol{\nu}}_{\partial \tilde{K}_{\tau} }\rangle e^{\frac{|x|^2}{4}} d\mu \\
 &\quad\qquad =\int_{S \backslash \tilde{K}_{\tau} }\left( \tilde{H} + \frac{\langle \tilde{X}, \tilde{\boldsymbol{\nu}}\rangle}{2}\right) e^{\frac{|x|^2}{4}} d\mu \,\geq \, 0\,.
\end{align*}
This proves the claim.
\end{proof}

\begin{proof}[Proof of Claim \ref{thm_one_sided}]
Note first that there exists a uniform constant $\sigma=\sigma(\alpha,\beta)>0$ such that $F=\langle X,\boldsymbol{\nu}\rangle +2tH \geq \sigma$ for any $t\in [0,\sigma]$
and such that $t_j\geq \sigma$ for $j$ large (since $\sup_{P(p_j,t_j,j^{-1}r_j)}|A|\to\infty$).

Since $\{\partial\tilde{K}_{\tau}\}_{\tau\in (-\infty, \log T)}$ foliates $\R^{n+1} \backslash \tilde{K}_{\log T}$ for any $T>0$ as long as the flow exits,
there is some $\Theta=\Theta(\alpha,\beta)<\infty$ such that $\tilde{K}_\tau$ is contained in $B_{\Theta}$ for any $\tau \geq \log \sigma$.
Moreover, by Claim \ref{thm_one_sided_rescaled}, $\tilde{K}_\tau$ is one-sided minimizing for the weighted area.
Also note that the rescaling factor between $\tilde{K}_\tau$ and $K_t$ is uniformly controlled for $\tau \geq \log \sigma$.

For any $\varepsilon >0$ there exists a constant $\delta=\delta(\varepsilon,\Theta)> 0$,
such that at any point $p\in \partial \tilde{K}_\tau$ ($\tau\geq \log\sigma$) we have
\begin{equation}\label{est_almostconstint}
1 - \varepsilon/2 \,\leq \,\frac{e^{\frac{|p|^2}{4}}\text{Area}(\partial \tilde{K}_{\tau} \cap B(p, \delta))}{\text{Area}_w(\partial \tilde{K}_{\tau} \cap B(p, \delta))} \,\leq \,1 + \varepsilon/2\,.
\end{equation}
Now at $p_j \in \partial K^j_{t}$, using the facts that $\mathcal{K}^j$ is $\alpha$-noncollapsed and that the parabolically rescaled flow $\hat{K}^j_t$ has height  $\lesssim \frac{D^2d}{\al} $ in the ball $B(0,R')$ where $R'$ is comparable  to $\sqrt{dR}$, we conclude that for any $r$ sufficiently small and $j$ sufficiently large:
\begin{align}\label{eqn_one_sided_m}
\text{Area}(\partial K_t^j \cap B(p_j, r)) \leq (1 + \varepsilon)\omega_n r^n\,.
\end{align}
Here, we used the estimate \eqref{est_almostconstint} and Claim \ref{thm_one_sided_rescaled} with $S$ obtained from $K^j_{t}$ by attaching a short solid cylinder over the approximate disk.
Rescaling to $\hat{\mathcal{K}}^j$ this completes the proof of Claim \ref{thm_one_sided}.
\end{proof}

\begin{remark}
\label{rem-forward_extension}
One may obtain a variant of the curvature estimate by considering
flows which are defined in $B(p,r)\times (t-r^2,t+\tau r^2]$ for some
fixed $\tau>0$, in which case the curvature bound holds in a suitable parabolic region extending forward in time. The proof is similar.
\end{remark}

%%%%%%%%%%%%%%%%%%%%%%%%%%%%%%%%%%%%%%%%%%%%%%%%%%%%%

\subsection{Convexity estimate}
In this section, we prove the following convexity estimate for mean curvature flow with star-shaped initial data.

\begin{theorem}\label{thm_convexity_est}
For all $\varepsilon>0$, there exists $\eta = \eta(\varepsilon, \alpha,\beta)<\infty$ with the following property. If $\mathcal{K}$ is a mean curvature flow with star-shaped initial condidition, defined in a parabolic ball $P(p,t,\eta r)$ centered at a boundary point $p \in \partial K_t$ with $H(p,t) \leq r^{-1}$, then
$$
\lambda_1 (p,t)\geq - \varepsilon r^{-1}\,.
$$
\end{theorem}

Theorem \ref{thm_convexity_est} immediately implies the following corollary.

\begin{corollary}
If $\mathcal{K}$ is a mean curvature flow with star-shaped initial condition,
then for all $\varepsilon>0$ there exists $0< H_0=H_0(\varepsilon,K_0)<\infty$ such that if $H(p,t)\geq H_0$ then $\frac{\lambda_1}{H}(p,t) \geq - \varepsilon$\,.
\end{corollary}

\begin{remark}
As mentioned in the introduction, a similar convexity estimate has been proved by Smoczyk \cite[Thm. 1.1]{Smoczyk_star_shaped} for $n=2$.
\end{remark}

\begin{proof}[Proof of Theorem \ref{thm_convexity_est}]
Fix $\alpha$ and $\beta$. We first show that the theorem holds for $\varepsilon \geq \frac{2D^2}{n\alpha}$, where $D$ is the diameter of $\partial K_0$.
To see this, we choose $\eta=\sqrt{n/2}$ and note that since the flow existed in the parabolic ball $P(p,t,\eta r)$ we have $(\eta r)^2 \leq t \leq \frac{D^2}{2n}$, c.f. \eqref{eqn_extinction_time}. Now the $\alpha$-noncollapsing (Theorem \ref{thm_noncollapsing}) gives interior and exterior balls of radius at least $\alpha/F(p,t)$ and thus
\begin{equation}\label{eqn_eigenvalue}
\lambda_1 (p,t) \geq - \frac{F(p,t)}{\alpha} \geq - \frac{D+ n^{-1}D^2r^{-1}}{\alpha}\geq - \varepsilon r^{-1},
\end{equation}
where we used that $\varepsilon \geq \frac{2D^2}{n\alpha}$ and $r\leq \frac{D}{n}$ by our choice of $\eta$. 

Let $\varepsilon_0 \leq  \frac{2D^2}{n\alpha}$ be the infimum of the $\varepsilon$'s for which the assertion of the theorem holds, and suppose towards a contradiction that $\varepsilon_0 >0$.

It follows that there is a sequence $\{\mathcal{K}^j\}$ defined in $P(p_j, t_j, \eta_j r_j)$ with $H(p_j, r_j) \leq r_j^{-1}$ and $\eta_j \to \infty$, but $\lambda_1(p_j, t_j) r_j \to - \varepsilon_0$. Now since
$$
(\eta_j r_j)^2 \leq t_j \leq \frac{D_j^2}{2n}
$$
is uniformly bounded, we have $r_j \to 0$. It follows that $\lambda_1(p_j, t_j) \to -\infty$.
Let $I:=\liminf_{j\to \infty}H(p_j, t_j)$. If $I<\infty$, then by the $\alpha$-noncollapsing condition we have
$$
\lambda_1(p_j, t_j) \geq - \frac{F(p_j, t_j)}{\alpha} \geq - \frac{D_j + D_j^2I/n}{\alpha}
$$
for some arbitrarily large integers $j$; a contradiction. Thus, $I=\infty$.

Parabolically rescaling by $r_j^{-1}$ and applying an isometry, we obtain a sequence $\{\hat{\mathcal{K}}^j\}$ of flows defined in $P(0,0, \eta_j)$ with $(0,0)\in \partial \hat{\mathcal{K}}^j, 0< H(0,0)\leq 1$ for all $j$, but $\lambda_1(0,0) \to -\varepsilon_0$ as $j\to \infty$. 
After passing to a subsequence, $\{\hat{\mathcal{K}}^j\}$  converges smoothly to a mean curvature flow $\hat{\mathcal{K}}^\infty$ in the parabolic ball $P(0,0,\rho)$, where $\rho=\rho(\alpha, \beta)$ is the quantity from Theorem \ref{thm-local_curvature_est}. For $\hat{\mathcal{K}}^\infty$ we have $\lambda_1(0,0)=-\eps_0$, and thus $H(0,0)=1$.
By continuity $H>\frac12$ in  $P(0,0,\rho')$ for some $\rho'\in (0,\rho)$.
Since $\varepsilon_0$ is the infimum of the $\varepsilon$'s for which the assertion of the theorem holds and since $I=\infty$,
it follow that $\tfrac{\lambda_1}{H}\geq -\varepsilon_0$ in $P(0,0,\rho')$.
Thus, $\tfrac{\lambda_1}{H}$ attains a negative minimum at $(0,0)$; this contradicts the strict maximum principle
(see e.g. \cite[App. A]{haslhofer-kleiner_mean_convex}, \cite[Sec. 8]{Hamilton_tensor_maximum} or \cite[App. A]{white_nature}).
\end{proof}

\subsection{Blowup theorem}
The next theorem shows that for mean curvature flow with star-shaped initial condition, we can pass to blowup limits smoothly and globally.

\begin{theorem}[Blowup theorem]\label{thm_blow_up_I}
Let $\K$ be a mean curvature flow with star-shaped initial condition. Let $\{(p_j,t_j)\in \D \K\}$
be a sequence of boundary points with $\lambda_j:=H(p_j, t_j) \to \infty$. 
Then, after passing to a subsequence, the flows  $\hat{\K}^j$ obtained from $\K$
by the rescaling $(p,t)\mapsto (\lambda_j(p-p_j),\lambda_j^2(t-t_j))$ converge smoothly and globally:
\begin{align}
 \hat{\K}^j\ra\K^\infty \qquad\qquad\qquad C^\infty_{\textrm{loc}} \,\, \textrm{on}\,\, \R^{n+1}\times(-\infty,0].
\end{align}
The limit $\K^\infty$ is a mean convex $\hat{\al}$-noncollapsed flow (i.e. admits interior and exterior balls of radius $\hat{\alpha}/H(p)$) for some $\hat{\alpha}=\hat{\alpha}(\alpha,\beta)>0$, and has convex time slices.
\end{theorem}

\begin{proof}

Since $\lambda_j = H(p_j, t_j) \to \infty$, we have that $t_j \geq \sigma>0$, for some uniform constant $\sigma>0$. By comparison with spheres, $t_j\leq T(\beta)<\infty$ where $\beta$ is from Proposition \ref{cor_H_lower_bound}.

There is a constant $\varepsilon>0$ and a sequence $\eta_j\to \infty$ such that the rescaled flow $\hat{\K}^j$ satisfies $\hat{H}(x,t)\geq \frac{\varepsilon}{1+\eta_j}$ in $P(0,0,\eta_j)$.
If not, the local curvature estimate (Theorem \ref{thm-local_curvature_est} and Remark \ref{rem-forward_extension}) centered at points with too small curvature yields $\hat{H}(0,0)<1$ for $j$ large; a contradiction.

By the above, the term $2tH$ is larger than $\langle X,\boldsymbol{\nu}\rangle$ in increasing parabolic neighborhood of the basepoint.
Therefore, $\hat{\K}^j$ is \emph{mean convex} $\hat{\alpha}$-noncollapsed in $P(0,0,\eta_j)$, where $\hat{\alpha}=\hat{\alpha}(\alpha,\beta)>0$. We can now apply the global convergence theorem \cite[Thm 1.12]{haslhofer-kleiner_mean_convex} to get that a limit $\K^\infty$, which is a mean convex $\hat{\alpha}$-noncollapsed flow with convex time slices.
\end{proof}

%%%%%%%%%%%%%%%%%%%%%%%%%%%%%%%%%%%%%%%%

\section{Estimates for weak solutions}\label{sec_weaksol}

\subsection{Elliptic regularization and consequences}

Let $K_0 \subset \R^{n+1}$ be a star-shaped domain and let $\{K_t\}_{t\geq 0}$ be the level set flow starting at $K_0$, see Section \ref{sec_prelim}.
By a result of Soner \cite[Sec. 9]{Soner}, the flow is nonfattening.
As in Section \ref{sec_mainest}, we consider the rescaled flow $\tilde{K}_\tau=t^{-1/2}K_t$ where $\tau=\log t$ and $t\in[\sigma,T]$.

We will now adapt the elliptic regularization from Evans-Spruck \cite[Sec. 7]{evans-spruck} to our setting. 
The rescaled level set flow $\{\tilde{K}_\tau\}_{\log \sigma \leq \tau\leq \log T}$
can be described by the time of arrival function ${v}: \tilde{K}_{\log \sigma} \to \R$ defined by $ (x_1,...,x_{n+1}) = \mathbf{x} \in \partial \tilde{K}_\tau\Leftrightarrow {v}(\mathbf{x})  +\log\sigma = \tau $.
The function ${v}$ satisfies
\begin{equation}\label{eqn_elliptic_regularization_1}
-\text{div} \left(\frac{D {v}}{|D {v}|}\right) - \frac{1}{2}{\left\langle \mathbf{x}, \frac{D {v}}{|D{v}|}\right\rangle} = \frac{1}{|D {v}|}\,,
\end{equation}
in the viscosity sense. The solution $v$ arises as uniform limit of smooth functions $v^{\varepsilon}: \tilde{K}_{\log \sigma} \to \R$
solving the regularized equation
\begin{equation}\label{eqn_elliptic_regularization_approx}
-\text{div} \left(\frac{(D {v^\varepsilon},-\varepsilon)}{\sqrt{\varepsilon^2+|D {v^\varepsilon}|^2}}\right) - \frac{1}{2}{\left\langle \left(\mathbf{x},x_{n+2}\right), \frac{(D {v^\varepsilon},-\varepsilon)}{\sqrt{\varepsilon^2+|D {v^\varepsilon}|^2}}\right\rangle} = \frac{1}{\sqrt{\varepsilon^2+|D {v^\varepsilon}|^2}}\,,
\end{equation}
with Dirichlet boundary conditions. Geometrically, equation \eqref{eqn_elliptic_regularization_approx} says that $\tilde{N}_{\log\sigma}^{\varepsilon} = \text{graph}\left(\frac{{v}^{\varepsilon} }{\varepsilon} \right)$ satisfies
\begin{equation}\label{eqn_elliptic_regularization_geom}
\vec{H} - \frac{X^{\perp} }{2} =  -\frac{1}{\varepsilon} \boldsymbol{e}_{n+2}^{\perp}\,,
\end{equation}
or equivalently that  $\tilde{N}_{\tau}^{\varepsilon} = \text{graph}\left(\frac{{v}^{\varepsilon}+ \log\sigma - \tau} {\varepsilon} \right), \tau \geq \log \sigma$,  is a translating solution of the rescaled mean curvature flow \eqref{eqn_rescaled_flow}.
Using a barrier argument as in  \cite[Sec. 7]{evans-spruck} we obtain the $C^0$-estimate
\begin{equation}
c\, \text{dist}(\mathbf{x}, \partial \tilde{K}_{\log \sigma}) \leq {v}^{\varepsilon}(\mathbf{x})  \leq c^{-1} \,\text{dist}(\mathbf{x}, \partial \tilde{K}_{\log \sigma})\,,
\end{equation}
for some uniform constant $c>0$. Multiplying by $\sqrt{\varepsilon^2+|D {v^\varepsilon}|^2}$ and taking the first partial derivative $D_{x_l}$ on both sides of equation \eqref{eqn_elliptic_regularization_approx} (replacing $x_{n+2}$ by $\frac{v^\varepsilon}{\varepsilon}$), we get
\begin{multline}\label{eqn_elliptic_regularization_deriv}
- \left(\delta_{ij} - \frac{{v}^{\varepsilon}_{x_i} {v}^{\varepsilon}_{x_j}}{\varepsilon^2 + |D {v}^{\varepsilon}|^2}\right) ({v}^{\varepsilon}_{x_l})_{x_i x_j} +  \frac{2 ({v}^{\varepsilon}_{x_l})_{x_i} {v}^{\varepsilon}_{x_j}}{\varepsilon^2 + |D {v}^{\varepsilon}|^2}  {v}^{\varepsilon}_{x_i x_j}\\
 -  \frac{2 {v}^{\varepsilon}_{x_i} {v}^{\varepsilon}_{x_j} {v}^{\varepsilon}_{x_k}({v}^{\varepsilon}_{x_l})_{x_k}}{\left(\varepsilon^2 + |D {v}^{\varepsilon}|^2\right)^2} {v}^{\varepsilon}_{x_i x_j} 
 -\frac{x_k({v}^{\varepsilon}_{x_l})_{x_k}}{2}  = 0\,.
\end{multline}
Thus, by the  maximum principle, we obtain the Lipschitz estimate
\begin{equation}
|D {v}^{\varepsilon}| \leq C,
\end{equation}
for some uniform constant $C<\infty$.
Therefore, as $\varepsilon$ tends to zero the functions ${v}^{\varepsilon}$ indeed converge uniformly to ${v}$, and $v$ is Lipschitz.

Now for $(\mathbf{x},x_{n+2})\in \tilde{N}_{\tau}^\varepsilon$ we have $\tau = v^\varepsilon(\mathbf{x}) +\log \sigma -\varepsilon x_{n+2}$. Thus, the time of arrival function of $\{\tilde{N}^\varepsilon_\tau\}$ is given by
\begin{equation}
 V^\varepsilon (\mathbf{x},x_{n+2}) =  v^\varepsilon(\mathbf{x})+ \log \sigma - \varepsilon x_{n+2}
\end{equation}
For $\varepsilon\to 0$ it converges locally uniformly to $V(\mathbf{x},x_{n+2}) = v(\mathbf{x}) + \log \sigma$, which is
the time of arrival function of $\{\partial \tilde{K}_\tau\times \R\}$.
Thus, for $\varepsilon\to 0$ the space-time tracks $\tilde{\mathcal{N}}^\varepsilon$ Hausdorff converge to $\tilde{\mathcal{K}}$, and similarly for their complements.
Together with Lemma \ref{lem_elliptic_approximator} below, we can now finish the argument as in \cite[Sec. 4.3]{haslhofer-kleiner_mean_convex} to conclude that the estimates from Section \ref{sec_mainest} hold for the level set flow with star-shaped initial condition, provided the mean curvature is interpreted in the viscosity sense:

\begin{definition}[{\cite[Def. 1.3]{haslhofer-kleiner_mean_convex}}]\label{def_viscosity}
Let $K\subseteq \R^{n+1}$ be a closed set. If $p\in \D K$, then the {\em viscosity mean curvature of $K$ at $p$} is
$$
H(p)=\inf\{H_{\D X}(p)\mid X\subseteq K\;\text{is a compact smooth domain,}
\;  p\in \D X\},\\
$$
where $H_{\D X}(p)$ denotes the mean curvature of $\D X$ at $p$ with respect to the inward pointing normal (here $\inf\emptyset=-\infty$).
\end{definition}

\begin{lemma}[c.f. {\cite[Thm. 4.6 (1)]{haslhofer-kleiner_mean_convex}}]\label{lem_elliptic_approximator}
The elliptic approximators $\tilde{{N}}_\tau^\varepsilon$ admit interior and exterior balls of radius at least
$\frac{\alpha_{\varepsilon}}{e\sigma} \sqrt{\varepsilon^2+|D {v^\varepsilon}(\mathbf{x})|^2}$ at $\tilde{X}^\varepsilon(\mathbf{x},\tau) = \left(\mathbf{x}, \frac{v^\varepsilon(\mathbf{x}) + \log\sigma-\tau}{\varepsilon}\right) \in \tilde{{N}}_\tau^\varepsilon$, and $\liminf_{\varepsilon \to 0} \alpha_{\varepsilon} \geq \alpha$.
\end{lemma}
\begin{proof}
As in the proof of Theorem \ref{thm_noncollapsing}, consider
$$\tilde{Z}^\varepsilon(\mathbf{x},\mathbf{y},\tau)= \frac{2\left\langle \tilde{X}^\varepsilon(\mathbf{y},\tau) - \tilde{X}^\varepsilon(\mathbf{x},\tau),  \tilde{\boldsymbol{\nu}}^\varepsilon(\mathbf{x},\tau) \right\rangle}{\|  \tilde{X}^\varepsilon(\mathbf{y},\tau) -  \tilde{X}^\varepsilon(\mathbf{x},\tau)\|^2}$$
and
$$ \tilde{Z}^\varepsilon_\ast(\mathbf{x},\tau)= \inf_{(\mathbf{y}\neq (\mathbf{x}}  \tilde{Z}^\varepsilon(\mathbf{x},\mathbf{y},\tau)\,,\quad  \tilde{Z}_\varepsilon^\ast (\mathbf{x},\tau)= \sup_{\mathbf{y}\neq \mathbf{x}}  \tilde{Z}^\varepsilon(\mathbf{x},\mathbf{y},\tau)\,,$$
where $\mathbf{x} \in \tilde{K}_{\log \sigma}$. Here
\begin{equation}\label{eqn_Z}
 \tilde{\boldsymbol{\nu}}^\varepsilon (\mathbf{x}, \tau) =(-D {v^\varepsilon}(\mathbf{x}),\varepsilon)/\sqrt{\varepsilon^2+|D {v^\varepsilon}(\mathbf{x})|^2}
 \end{equation}
and 
\begin{equation}\label{eqn_X}
 \tilde{X}^\varepsilon(\mathbf{x},\tau) = \left(\mathbf{x}, \frac{v^\varepsilon(\mathbf{x}) + \log\sigma-\tau}{\varepsilon}\right) \in \tilde{{N}}_\tau^\varepsilon\,.
  \end{equation}
Since $\tilde{{N}}_\tau^\varepsilon$ is a translating solution of the rescaled mean curvature flow \eqref{eqn_rescaled_flow}, we denote ${N}_\tau^\varepsilon = e^{\tau/2}\tilde{{N}}_\tau^\varepsilon$ the mean curvature flow corresponding to $\tilde{{N}}_\tau^\varepsilon$ and $X^\varepsilon(\mathbf{x},t) = e^{\tau/2}\tilde{X}^\varepsilon(\mathbf{x},\tau)$, where $\tau = \log t$. Let 
$$
\tilde{F}^\varepsilon(\mathbf{x},\tau) = \tilde{F}^\varepsilon(\tilde{X}^\varepsilon(\mathbf{x},\tau))  = \tilde{H}^{\varepsilon}( \tilde{X}^\varepsilon(\mathbf{x},\tau)) +   \frac{\left\langle  \tilde{X}^\varepsilon(\mathbf{x},\tau) , \tilde{\boldsymbol{\nu}}^\varepsilon(\mathbf{x},\tau)\right\rangle}{2}
$$
and
\begin{equation}\label{eqn_F_10}
 F^\varepsilon (\mathbf{x},t)= 2 e^{\tau/2}\tilde{F}^\varepsilon (\mathbf{x},\tau)=  2tH^{\varepsilon}( X^\varepsilon(\mathbf{x},t)) + \left\langle  X^\varepsilon(\mathbf{x},t) , \boldsymbol{\nu}^\varepsilon(\mathbf{x},t)\right\rangle\,,
\end{equation}
cf. \eqref{eqn_rescaled_flow_mono}. 

Similarly, we define 
\begin{equation}\label{Z_ast}
Z^\varepsilon_\ast(\mathbf{x},t) = e^{-\tau/2} \tilde{Z}^\varepsilon_\ast(\mathbf{x},\tau) \quad \text{and}\quad Z_\varepsilon^\ast(\mathbf{x},t) = e^{-\tau/2} \tilde{Z}_\varepsilon^\ast(\mathbf{x},\tau)
\end{equation}
according to the rescaling\,. Then we have
\begin{equation}\label{eqn_ZoverF}
\frac{Z^\varepsilon_\ast}{F^\varepsilon}= \frac{\tilde{Z}^\varepsilon_\ast}{2e^{\tau}\tilde{F}^\varepsilon}\,.
\end{equation}
Now note that equation \eqref{eqn_elliptic_regularization_approx} is equivalent to
\begin{equation}\label{eqn_F_init}
\tilde{F}^\varepsilon (\mathbf{x}, \log\sigma) = 1/ \sqrt{\varepsilon^2+|D {v^\varepsilon}(\mathbf{x})|^2}\quad \text{for }\forall \,\,\mathbf{x} \in \tilde{K}_{\log \sigma}\,.
\end{equation}
Moreover, since $\tilde{{N}}_\tau^\varepsilon$ is a translating solution of the rescaled mean curvature flow \eqref{eqn_rescaled_flow} (so that for fixed $\mathbf{x}$ we know that $\tilde{Z}^\varepsilon_\ast (\mathbf{x},\tau), \tilde{Z}_\varepsilon^\ast (\mathbf{x},\tau)$ and $\tilde{H}^\varepsilon(\mathbf{x},\tau)$ are independent of $\tau$), using \eqref{eqn_Z} and \eqref{eqn_X} we have
$$
\frac{d}{d\tau}\tilde{Z}^\varepsilon_\ast (\mathbf{x},\tau)= 0
$$
and
\begin{equation}\label{eqn_F_varepsilon_2}
\frac{d}{d\tau} \tilde{F}^\varepsilon (\mathbf{x},\tau)=  \frac{1}{2}\frac{d}{d\tau}   \left\langle  \tilde{\boldsymbol{\nu}}^\varepsilon , \tilde{X}^\varepsilon\right\rangle (\mathbf{x},\tau)=  \frac{-1}{2\sqrt{\varepsilon^2+|D {v^\varepsilon}(\mathbf{x})|^2} }\,.
\end{equation}
Therefore, integrating \eqref{eqn_F_varepsilon_2} w.r.t. $\tau$ and using  \eqref{eqn_F_init} we have
\begin{equation}\label{eqn_F_vareps_1}
\tilde{F}^\varepsilon \left(\mathbf{x}, \tau \right) = \left (\frac{2+\log\sigma - \tau}{2}\right)/\sqrt{\varepsilon^2+|D {v^\varepsilon}(\mathbf{x})|^2}\,.
\end{equation}
Therefore
\begin{align}\label{eqn_F_varepsilon}
0= &\,\frac{d}{d\tau} \left(\frac{\left (2+\log\sigma - \tau\right) \tilde{Z}^\varepsilon_\ast}{2\tilde{F}^\varepsilon} \right) (\mathbf{x},\tau)\\
=&\,  \partial_\tau \left(\frac{\left (2+\log\sigma - \tau\right) \tilde{Z}^\varepsilon_\ast}{2\tilde{F}^\varepsilon} \right) (\mathbf{x},\tau) + \left\langle \tilde{\nabla } \left(\frac{\left (2+\log\sigma - \tau\right) \tilde{Z}^\varepsilon_\ast}{2\tilde{F}^\varepsilon} \right),  \frac{\mathbf{e}_{n+2}^{\text{T}}}{\varepsilon|  \mathbf{e}_{n+2}^{\text{T}}|}\right\rangle(\mathbf{x},\tau)\,,\notag
 \end{align}
where $  \mathbf{e}_{n+2}^{\text{T}}$ is the tangential part of $  \mathbf{e}_{n+2}$ at $\tilde{X}^\varepsilon(\mathbf{x},\tau)$ and $\partial_\tau \left(\frac{\left (2+\log\sigma - \tau\right) \tilde{Z}^\varepsilon_\ast}{2\tilde{F}^\varepsilon} \right)$  is the time derivative of $\frac{(2+\log\sigma - \tau ) \tilde{Z}^\varepsilon_\ast}{2\tilde{F}^\varepsilon} $ along the normal motion.

Now using \eqref{eqn_Z_ast}, \eqref{eqn_ALM} and \eqref{eqn_ZoverF} we obtain
 \begin{equation}\label{eqn_evo_ZF}
 \partial_\tau \frac{\tilde{Z}^\varepsilon_\ast}{\tilde{F}^\varepsilon} \geq \tilde{\Delta}\frac{\tilde{Z}^\varepsilon_\ast}{\tilde{F}^\varepsilon}+2 \left\langle  \tilde{\nabla} \log \tilde{F}^\varepsilon,  \tilde{\nabla}\frac{\tilde{Z}^\varepsilon_\ast}{\tilde{F}^\varepsilon}\right\rangle +\frac{\tilde{Z}^\varepsilon_\ast}{\tilde{F}^\varepsilon} \,,
 \end{equation}
 in the viscosity sense. Combining \eqref{eqn_F_varepsilon} and \eqref{eqn_evo_ZF} we obtain
 \begin{equation}\label{eqn_evo_ZF_1}
 0
\geq \frac{2+\log\sigma - \tau}{2}\left( \tilde{\Delta}\frac{\tilde{Z}^\varepsilon_\ast}{\tilde{F}^\varepsilon}+2 \left\langle  \tilde{\nabla} \log \tilde{F}^\varepsilon + \frac{\mathbf{e}_{n+2}^{\text{T}}}{2\varepsilon|  \mathbf{e}_{n+2}^{\text{T}}|},  \tilde{\nabla}\frac{\tilde{Z}^\varepsilon_\ast}{\tilde{F}^\varepsilon}\right\rangle\right) +\frac{(1+ \log\sigma -\tau )\tilde{Z}^\varepsilon_\ast}{2\tilde{F}^\varepsilon}\,, \end{equation}
 if $\tau <2 + \log\sigma$. Note that $2+\log\sigma - \tau>0$ and $1+ \log\sigma -\tau  \leq 0$ if $\tau \in [1+\log\sigma, 2+\log\sigma)$.

Therefore, by \eqref{eqn_F_varepsilon}, the quantity
\begin{equation}\label{eqn_noncollapsing-constant}
I_\varepsilon(\tau):=\min_{\tilde{N}^\varepsilon_\tau}\frac{\tilde{Z}^\varepsilon_\ast}{2\tilde{F}^\varepsilon/(2+\log\sigma - \tau)} 
\end{equation}
and the value of the optimal noncollapsing constant $\alpha_\varepsilon$ of $\tilde{N}^\varepsilon_\tau$ (with respect to the radius 
$\frac{1}{2\tilde{F}^\varepsilon/(2+\log\sigma - \tau)}$) are independent of time $\tau\geq \log\sigma$.
Moreover, at any time $\tau \in [1+\log\sigma, 2+\log\sigma)$ we can apply the maximum principle to equation \eqref{eqn_evo_ZF_1}  so that we know $I_\varepsilon(\tau)$ is attained at the boundary of $\tilde{N}^\varepsilon_\tau$.  Since $\{\tilde{N}^\varepsilon_\tau\}_{\tau \geq \log\sigma}$ converges locally uniformly to $\{\tilde{K}_\tau \times \R\}_{\tau\geq \log\sigma}$ as $\varepsilon \to 0$ (and the convergence is smooth at least until $\tau = 2 + \log\sigma$ if $\sigma$ is chosen sufficiently small), to find the limiting behavior of the noncollapsing constant as $\varepsilon \to 0$, we can simply look any time $\tau = 1 +\log\sigma$ to conclude that  (note also that $\tilde{K}_{1+\log\sigma}$ admits interior and exterior balls of radius at least $\alpha/(2e\sigma \tilde{F})$ where $\tilde{F} =  \tilde{H} + \frac{\langle\tilde{X}, \tilde{\boldsymbol{\nu}}\rangle}{2}$)
\begin{equation}\label{eqn_noncollasing_constant}
\liminf_{\varepsilon\to 0}I_\varepsilon \geq -\frac{e\sigma }{\alpha}.
\end{equation}

Therefore, using \eqref{eqn_F_init} and \eqref{eqn_F_vareps_1} we know that $\tilde{{N}}_\tau^\varepsilon$ admits interior balls of radius at least
\begin{equation}\label{eqn_rad}
\frac{ \alpha_{\varepsilon}/ (e\sigma)}{2\tilde{F}^\varepsilon(\mathbf{x},\tau)/(2+\log\sigma - \tau)} = \frac{\alpha_{\varepsilon} }{ e\sigma\tilde{F}^{\varepsilon}(\mathbf{x}, \log\sigma)} = \frac{\alpha_\varepsilon \sqrt{\varepsilon^2+|D {v^\varepsilon}(\mathbf{x})|^2}}{e\sigma}
\end{equation}
at $\tilde{X}^{\varepsilon}(\mathbf{x},\tau)$ for all $\mathbf{x}\in \tilde{K}_{\log\sigma}$ and all $\tau\geq \log\sigma$. Moreover, $\liminf_{\varepsilon \to 0} \alpha_{\varepsilon} \geq \alpha$.
Arguing similarly for $\tilde{Z}_\varepsilon^\ast$, this proves the lemma.
\end{proof}

\begin{remark} \label{remark_1}
To see that $\tilde{K}_{1+\log\sigma}$ admits interior and exterior balls of radius at least $\alpha/(2e\sigma \tilde{F})$ where $\tilde{F} =  \tilde{H} + \frac{\langle\tilde{X}, \tilde{\boldsymbol{\nu}}\rangle}{2}$, we note that if $N_t = e^{\tau/2}\tilde{N}_\tau$ admits interior and exterior balls of radius at least $\alpha/F$ at $X(\mathbf{x},t)$, then by the rescaling $\tilde{X}(\mathbf{x},\tau) = t^{-1/2}X(\mathbf{x},t)$ (cf. \eqref{eqn_F_10}) we know that $\tilde{N}_\tau$ admits interior and exterior balls of radius at least $\alpha/(2e^{\tau}\tilde{F})$
at $\tilde{X}(\mathbf{x},\tau)$.
\end{remark}

\begin{remark}\label{remark_2}
From \eqref{eqn_noncollapsing-constant} and \eqref{eqn_rad} we see that the noncollapsing constant of $\tilde{N}^\varepsilon_\tau$, with respect to the radius 
$$
\frac{1}{2\tilde{F}^\varepsilon/(2+\log\sigma - \tau)} =  \sqrt{\varepsilon^2+|D {v^\varepsilon}(\mathbf{x})|^2}\,,
$$
is at least $\alpha/(e\sigma) = \frac{2\alpha/e}{2e^{\log\sigma}}$ for all $\tau\geq \log\sigma$.

Therefore, using \eqref{eqn_F_10}, \eqref{Z_ast} and \eqref{eqn_noncollapsing-constant} we know that
\begin{equation}
I_\varepsilon(t): = \min_{N^\varepsilon_t}\frac{t Z^\varepsilon_\ast}{F^\varepsilon/(2+\log\sigma - \log t)} 
\end{equation}
 is independent of $t\geq \sigma$. Namely, the noncollapsing constant of  $N^\varepsilon_t = e^{\tau/2} \tilde{N}^\varepsilon_\tau$, with 
respect to the radius
$$
\frac{t}{F^\varepsilon/(2+\log\sigma - \log t)} =  \sqrt{t}\sqrt{\varepsilon^2+|D {v^\varepsilon}(\mathbf{x})|^2}\,,
$$
is independent of $t$. 

Now take $t=\sigma$. Since the noncollapsing constant of $\tilde{N}^\varepsilon_{\log\sigma}$  (w.r.t. $ 1/\tilde{F}^{\varepsilon}$) is at least $\alpha/(e\sigma)$, by the same rescaling as in Remark \ref{remark_1} we know that the noncollapsing constant of  $N^\varepsilon_\sigma$ (w.r.t. $ \sqrt{\sigma}\sqrt{\varepsilon^2+|D {v^\varepsilon}(\mathbf{x})|^2} = 2\sigma/F^{\varepsilon}$) is at least $2\alpha/e$, and thus the noncollapsing constant of  $N^\varepsilon_t$ (w.r.t. $\sqrt{t}\sqrt{\varepsilon^2+|D {v^\varepsilon}(\mathbf{x})|^2}$) is at least $2\alpha/e$ for all $t\geq \sigma$ since it is independent of $t$.
\end{remark}

\begin{remark}
Applying Lemma \ref{lem_elliptic_approximator} to $\{\tilde{N}^\varepsilon_\tau\}_{\tau \geq \log\sigma}$ and by Remark \ref{remark_2} we know that the noncollapsing constant of $\{N^\varepsilon_t\}_{t \geq \sigma}$ (w.r.t. $\sqrt{t}\sqrt{\varepsilon^2+|D {v^\varepsilon}(\mathbf{x})|^2}$) is at least $2\alpha/e$ for all $t\geq \sigma$. Since $\{\tilde{N}^\varepsilon_\tau\}_{\tau \geq \log\sigma}$ and $\{N^\varepsilon_t\}_{t \geq \sigma}$ converges locally uniformly to $\{\tilde{K}_\tau \times \R\}_{\tau\geq \log\sigma}$ and $\{K_t \times \R\}_{t\geq\sigma}$, respectively, as $\varepsilon \to 0$, we get that the noncollapsing constant of $K_t$ (w.r.t. $\lim_{\varepsilon\to0}\sqrt{t}\sqrt{\varepsilon^2+|D {v^\varepsilon}(\mathbf{x})|^2} = e^{\tau/2}/\tilde{F} =\frac{1}{F/(2t)} = \frac{1}{H+\langle X, \boldsymbol{\nu}\rangle/(2t) } >0$) is at least $2\alpha/e$ for all $t\geq \sigma$.
\end{remark}

%%%%%%%%%%%%%%%%%%%%%%%%%%%%%%%%%%%%%%%%%%%%%%%%%%%%%%%%
\subsection{Size and structure of the singular set}
In this final section we describe the size and the structure of the singular set for the mean curvature flow with star-shaped initial condition.

\begin{theorem}[Tangent flows]\label{thm_tangent_flow}
Let $\K$ be a mean curvature flow with star-shaped initial condition. Let $(p,t) \in \D \K$ ($t>0$) and let $\lambda_j\to \infty$. 
Then, the flow $\K^j$ obtained from $\K$
by the parabolic rescaling $(p,t)\mapsto (\lambda_j(p-p_j),\lambda_j^{2}(t-t_j))$ converges smoothly and globally:
\begin{align}
\K^j\ra\K^\infty \qquad\qquad\qquad C^\infty_{\textrm{loc}} \,\, \textrm{on}\,\, \R^{n+1}\times(-\infty,0].
\end{align}
The limit $\K^\infty$ is either (i) a static halfspace or (ii) a shrinking round sphere or cylinder.
\end{theorem}

\begin{proof}
Let $Q_j:=\sup_{\K^j\cap P(0,0,1)} H$.
If there is a subsequence such that $Q_j\lambda_j^{-1}\to 0$, then by the local curvature estimate (Theorem \ref{thm-local_curvature_est}) we have convergence to a static halfspace.
Assume now $\liminf_{j\to\infty} Q_j\lambda_j^{-1}>0$.
Then, arguing as in the proof of the blowup theorem (Theorem \ref{thm_blow_up_I}) we see that $\K^j$ is \emph{mean convex} $\hat{\alpha}$-noncollapsed in $P(0,0,\eta_j)$ for some sequence $\eta_j\to \infty$.
Applying the structure theorem \cite[Thm. 1.14]{haslhofer-kleiner_mean_convex} we conclude that a subsequence converges to a round shrinking sphere or cylinder.
Finally, by a recent result of Colding-Minicozzi \cite{CM_unique} the limit is unique, i.e. we have convergence even without passing to a subsequence.
\end{proof}

\begin{theorem}[Partial regularity]\label{thm_partial_regularity}
Suppose $\mathcal{K}$ is a mean curvature flow with star-shaped initial condition. Then the parabolic Hausdorff dimension and Minkowski dimension of the singular set $\mathcal{S}\subset \R^{n+1,1}$ are at most $n-1$. Moreover, $\mathcal{H}_{\textrm{par}}^{n-1}(\mathcal{S})<\infty$.
\end{theorem}

\begin{proof}
The estimate for the parabolic Hausdorff dimension is a quick consequence of the tangent flow theorem (Theorem \ref{thm_tangent_flow}).
Namely, if the parabolic Hausdorff dimension of $\mathcal{S}$ where bigger than $n-1$,
then blowing up at a density point we would obtain a tangent flow whose singular set has parabolic Hausdorff dimension bigger than $n-1$, contradicting the classification of tangent flows.
The stronger estimate for the parabolic Minkowski dimension and the finiteness of $\mathcal{H}_{\textrm{par}}^{n-1}(\mathcal{S})$
can be obtained by combining Theorem \ref{thm_tangent_flow} with the work of Cheeger-Haslhofer-Naber \cite{CHN} and Colding-Minicozzi \cite{CM_singular_set}, respectively.
\end{proof}

\bibliographystyle{alpha}
\bibliography{Starshaped-MCF}

\end{document}